\documentclass[runningheads]{llncs}
\usepackage{graphicx}
\usepackage{url}

\usepackage{hyperref}

\usepackage{amssymb,amsmath, amsrefs}
 
\newcommand{\Z}{{\textsf{\textup{Z}}}}

\newtheorem{thm}{Theorem}

\newtheorem{cor}[thm]{Corollary}
\newtheorem{defi}[thm]{Definition}
\newtheorem{rem}[thm]{Remark}
\newtheorem{nota}[thm]{Notation}

\newtheorem{princ}[thm]{Principle}

\newcommand\be{\begin{equation}}
\newcommand\ee{\end{equation}} 

\usepackage{amsmath,amsfonts}

\def\bdefi{\begin{defi}\rm}
\def\edefi{\end{defi}}
\def\bnota{\begin{nota}\rm}
\def\enota{\end{nota}}
\def\r{\mathfrak{r}}

\def\FIVE{\Pi_{1}^{1}\text{-\textup{\textsf{CA}}}_{0}}

\def\SIXK{\Pi_{k}^{1}\text{-\textsf{\textup{CA}}}_{0}^{\omega}}

\def\Y{\textup{\textsf{Y}}}
\def\bv{\textup{\textsf{bv}}}
\def\semi{\textup{\textsf{semi}}}

\def\cliq{\textup{\textsf{cliq}}}
\def\ATR{\textup{\textsf{ATR}}}

\def\ZFC{\textup{\textsf{ZFC}}}
\def\ZF{\textup{\textsf{ZF}}}

\def\L{\textsf{\textup{L}}}

\def\RCA{\textup{\textsf{RCA}}}
\def\({\textup{(}}
\def\){\textup{)}}

\def\RCAo{\textup{\textsf{RCA}}_{0}^{\omega}}
\def\ACAo{\textup{\textsf{ACA}}_{0}^{\omega}}

\def\WKL{\textup{\textsf{WKL}}}

\def\WWKL{\textup{\textsf{WWKL}}}
\def\bye{\end{document}}
\def\N{{\mathbb  N}}
\def\Q{{\mathbb  Q}}
\def\R{{\mathbb  R}}

\def\SS{\textup{\textsf{S}}}

\def\di{\rightarrow}

\def\asa{\leftrightarrow}

\def\ACA{\textup{\textsf{ACA}}}

\def\AC{\textup{\textsf{AC}}}

\def\INDY{\textup{\textsf{IND}}_{0}}

\def\Borel{\textup{\textsf{Borel}}}

\def\cocode{\textup{\textsf{cocode}}}

\def\NIN{\textup{\textsf{NIN}}}

\def\NBI{\textup{\textsf{NBI}}}

\def\IND{\textup{\textsf{IND}}}
\def\NFP{\textup{\textsf{NFP}}}

\def\HBU{\textup{\textsf{HBU}}}

\def\eps{\varepsilon}

\def\X{\textup{\textsf{X}}}

\usepackage{comment}

\begin{document}
\title{Reverse Mathematics of the uncountability of $\R$\thanks{This research was supported by the \emph{Deutsche Forschungsgemeinschaft} (DFG) via the grant \emph{Reverse Mathematics beyond the G\"odel hierarchy} (SA3418/1-1).  I thank Ulrich Kohlenbach and Dag Normann for all helpful advise regarding Section \ref{basic}.  I also thank the anonymous referees for their many helpful suggestions.}}
%
%
\author{Sam Sanders\inst{1}}
\authorrunning{S.\ Sanders}
%
\institute{Department of Philosophy II, RUB Bochum, Germany \\
\email{sasander@me.com} \\
\url{https://sasander.wixsite.com/academic}}

\setcounter{secnumdepth}{3}
\setcounter{tocdepth}{3}

%
\maketitle              
\begin{abstract}
In his first set theory paper (1874), Cantor establishes the uncountability of $\R$.
We study the latter in Kohlenbach's \emph{higher-order} Reverse Mathematics, motivated by the 
observation that one cannot study concepts like `arbitrary mappings from $\R$ to $\N$' in second-order Reverse Mathematics.
Now, it was recently shown that the statement
\[
\textup{$\NIN:$ \emph{there is no injection from $[0,1]$ to $\N$}}
\]
is \emph{hard to prove} in terms of conventional comprehension.  
In this paper, we show that $\NIN$ is \emph{robust} by establishing equivalences between $\NIN$ and $\NIN$ restricted to mainstream function classes, like: bounded variation, semi-continuity, and Borel.  
Thus, the aforementioned hardness of $\NIN$ is \textbf{not} due to the quantification over \emph{arbitrary} $\R\di \N$-functions in $\NIN$.       
Finally, we also study $\NBI$, the restriction of $\NIN$ to \emph{bijections}, and the connection to Cousin's lemma and Jordan's decomposition theorem.  
\end{abstract}

\section{Introduction and preliminaries}\label{intro}
\subsection{Aim and motivation}\label{krum}

In a nutshell, we study the \emph{the uncountability of $\R$} from the point of view of \emph{Reverse Mathematics}.  
We now explain the aforementioned italicised notions.  

\smallskip

First of all, Reverse Mathematics (RM hereafter) is a program in the foundations of mathematics initiated by Friedman (\cites{fried, fried2}) and developed extensively by Simpson and others (\cites{simpson1, simpson2}); an introduction to RM for the `mathematician in the street' is in \cite{stillebron}.  In a nutshell, RM seeks to identify the minimum axioms needed to prove theorems of ordinary, i.e.\ non-set theoretic, mathematics. 
We assume basic familiarity with RM, including Kohlenbach's \emph{higher-order} RM introduced in \cite{kohlenbach2}, with more recent results -including our own- in \cites{dagsamV, dagsamIII, dagsamX, dagsamVII, dagsamXI, yamayamaharehare}.

\smallskip

Now, the biggest difference between `classical' RM and higher-order RM is that the former makes use of $L_{2}$, the language of \emph{second-order} arithmetic, while the latter uses $L_{\omega}$, the language of \emph{higher-order} arithmetic.  Thus, higher-order objects are only indirectly available via so-called codes or representations in classical RM.  In particular, $L_{2}$ cannot talk about `arbitrary mappings from $\R$ to $\N$'.  Thus, Simpson (only) proves that the real numbers $\R$ cannot be enumerated as a sequence in classical RM (see \cite{simpson2}*{II.4.9}).  Hence, the higher-order RM of the \emph{uncountability of $\R$}, discussed next, is a natural (wide-open) topic of study.  

\smallskip

Secondly, the uncountability of $\R$ was established in 1874 by Cantor in his \emph{first} set theory paper \cite{cantor1}, which even has its own Wikipedia page, namely \cite{wica}.  We will study the uncountability of $\R$ in the guise of the following principles:
\begin{itemize}
\item $\NIN$: \emph{there is no injection from $[0,1]$ to $\N$},
\item  $\NBI$: \emph{there is no bijection from $[0,1]$ to $ \N$}.
\end{itemize}
It was established in \cite{dagsamXI} that $\NIN$ and $\NBI$ are \emph{hard} to prove in terms of (conventional) comprehension, as explained in detail in Remark \ref{proke}.  
One obvious way of downplaying these results is to simply attribute the hardness of $\NIN$ to the fact that one quantifies over \emph{arbitrary} third-order objects, namely $\R\di\N$-functions.  

\smallskip

In this paper, we establish RM-equivalences involving $\NIN$ and $\NBI$, where some are straightforward (Section \ref{basic}) and others advanced or surprising (Section \ref{adv}).  
We also study the connection between $\NIN$ and \emph{Cousin's lemma} and \emph{Jordan's decomposition theorem} (Section \ref{knew}).
In particular, we show that $\NIN$ is equivalent to the statement that there is no injection from $[0,1]$ to $\Q$ that enjoys `nice' mainstream properties 
like \emph{bounded variation}, \emph{semi-continuity}, and related notions.
Hence, the aforementioned hardness of $\NIN$ and $\NBI$ is not due to the latter quantifying over arbitrary third-order functions as \emph{exactly} the same hardness is observed for \emph{mathematically natural} subclasses.  
A recent FOM-discussion initiated by Friedman via \cite{fomo}, brought about this insight, while our results establish that $\NIN$ is \emph{robust} in the sense of Montalb\'an, as follows.
\begin{quote}
[\dots] gaining a greater understanding of [the big five] phenomenon is currently one of the driving questions behind reverse mathematics.
To study the big five phenomenon, one distinction that I think is worth making is the one between robust systems and non-robust systems. A system is \emph{robust} if it is equivalent to small perturbations of itself. This is not a precise notion yet, but we can still recognize some robust systems. All the big five systems are very robust. [...] Apart from those systems, weak weak K\"onig's Lemma ($\WWKL_{0}$) is also robust, and we know no more than one or two other systems that may be robust. (\cite{montahue}*{p.\ 432})
\end{quote} 
Thirdly, as to the structure of this paper, we introduce some essential axioms and definitions in Section \ref{prelim} while our main results may be found in Section \ref{main}.  
We note that some of our results are proved using $\IND_{0}$, a non-trivial fragment of the induction axiom from Section \ref{prelim1}. 
It is a natural RM-question, posed previously by Hirschfeldt (see \cite{montahue}*{\S6.1}), whether these extra axioms are needed for the reversal.
Neeman provides an example of the necessary use of extra induction in a reversal in\cite{neeman}. 
We finish this introductory section with a conceptual remark.  
\begin{rem}[Conventional comprehension]\label{proke}\rm
First of all, the goal of RM is to find the minimal axioms that prove a given theorem.  In second-order RM, these minimal axioms are fragments of the comprehension axiom (and related notions), i.e.\ 
the statement that the set $\{n\in \N: \varphi(n)\}$ exists for a certain class of $L_{2}$-formulas.  Higher-order RM similarly makes use of `comprehension functionals', i.e.\ \emph{third-order} objects
that decide formulas in a certain sub-class of $L_{2}$.  Examples include Kleene's quantifier $\exists^{2}$ and the Suslin functional $\SS^{2}$, to be found in Section \ref{prelim1}.  We are dealing with \emph{conventional} comprehension here, i.e.\
only first- and second-order objects are allowed as parameters.  

\smallskip

Secondly, second-order arithmetic $\Z_{2}$ has two natural higher-order formulations $\Z_{2}^{\omega}$ and $\Z_{2}^{\Omega}$ based on comprehension functionals, both to be found in Section \ref{prelim1}.  
The systems $\Z_{2}$, $\Z_{2}^{\omega}$, and $\Z_{2}^{\Omega}$ prove the same second-order sentences by \cite{hunterphd}*{Cor.\ 2.6}.  Nonetheless, the system $\Z_{2}^{\omega}$ \textbf{cannot} prove $\NIN$ or $\NBI$, while $\Z_{2}^{\Omega}$ proves both.  Here, $\Z_{2}^{\omega}$ and $\NIN$ can be formulated in the language of third-order arithmetic, i.e.\ there is no `type mismatch'.
The previous negative result is why we (feel obliged/warranted to) say that \emph{the principle $\NIN$ is hard to prove in terms of conventional comprehension}.
Finally, $\NIN$ and $\NBI$ seem to be the weakest natural third-order principles with this hardness property.  
\end{rem}

\subsection{Preliminaries}\label{prelim}
We introduce axioms and definitions from RM needed below.  We refer to \cite{kohlenbach2}*{\S2} or \cite{dagsamIII}*{\S2} for Kohlenbach's base theory $\RCAo$, 
and basic definitions like the real numbers $\R$ in $\RCAo$.  As in second-order RM (see \cite{simpson2}*{II.4.4}), real numbers are represented by fast-converging Cauchy sequences.
To avoid the details of coding real numbers and sets, we often assume the axiom $(\exists^{2})$ from Section \ref{prelim1}, which can however sometimes be avoided, as discussed in Remark \ref{LEM}.
\subsubsection{Some axioms of higher-order arithmetic}\label{prelim1}
First of all, the functional $\varphi$ in $(\exists^{2})$ is clearly discontinuous at $f=11\dots$; in fact, $(\exists^{2})$ is equivalent to the existence of $F:\R\di\R$ such that $F(x)=1$ if $x>_{\R}0$, and $0$ otherwise (\cite{kohlenbach2}*{\S3}).  
\be\label{muk}\tag{$\exists^{2}$}
(\exists \varphi^{2}\leq_{2}1)(\forall f^{1})\big[(\exists n)(f(n)=0) \asa \varphi(f)=0    \big]. 
\ee
Related to $(\exists^{2})$, the functional $\mu^{2}$ in $(\mu^{2})$ is also called \emph{Feferman's $\mu$} (\cite{kohlenbach2}).
\begin{align}\label{mu}\tag{$\mu^{2}$}
(\exists \mu^{2})(\forall f^{1})\big[ \big((\exists n)(f(n)=0) \di [f(\mu(f))=0&\wedge (\forall i<\mu(f))(f(i)\ne 0) ]\big)\\
& \wedge [ (\forall n)(f(n)\ne0)\di   \mu(f)=0]    \big].\notag
\end{align}
Intuitively, $\mu^{2}$ is the least-number-operator, i.e.\ $\mu(f)$ provides the least $n\in \N$ such that $f(n)=0$, if such number exists.  
We have $(\exists^{2})\asa (\mu^{2})$ over $\RCAo$ and $\ACAo\equiv\RCAo+(\exists^{2})$ proves the same $L_{2}$-sentences as $\ACA_{0}$ by \cite{hunterphd}*{Theorem~2.5}.
Working in $\ACAo$, one readily defines a
functional $\eta:[0,1]\di 2^{\N}$ that converts real numbers to their\footnote{In case there are two binary representations, we choose the one with a tail of zeros.} binary representation. 

\smallskip

Secondly, we sometimes need more induction than is available in $\RCAo$. 
The connection between `finite comprehension' and induction is well-known from second-order RM (see \cite{simpson2}*{X.4.4}).
\begin{princ}[$\INDY$]\label{bcktre}
Let $Y^{2}$ satisfy $(\forall n\in \N)(\exists \textup{ at most one } f\in 2^{\N})(Y(f, n)=0)$.  
For $k\in \N$, there is $w^{1^{*}}$ such that for any $m\leq k$, we have
\[
(\exists i<|w|)((w(i)\in 2^{\N}\wedge Y(w(i), m)=0) )\asa (\exists f\in 2^{\N})(Y(f, m)=0).
\]
\end{princ}
Thirdly, \emph{the Suslin functional} $\SS^{2}$ is defined in \cite{kohlenbach2} as follows:
\be\tag{$\SS^{2}$}
(\exists\SS^{2}\leq_{2}1)(\forall f^{1})\big[  (\exists g^{1})(\forall n^{0})(f(\overline{g}n)=0)\asa \SS(f)=0  \big].
\ee
The system $\FIVE^{\omega}\equiv \RCAo+(\SS^{2})$ proves the same $\Pi_{3}^{1}$-sentences as $\FIVE$ by \cite{yamayamaharehare}*{Theorem 2.2}.   
By definition, the Suslin functional $\SS^{2}$ can decide whether a $\Sigma_{1}^{1}$-formula as in the left-hand side of $(\SS^{2})$ is true or false.   We similarly define the functional $\SS_{k}^{2}$ which decides the truth or falsity of $\Sigma_{k}^{1}$-formulas from $\L_{2}$; we also define 
the system $\SIXK$ as $\RCAo+(\SS_{k}^{2})$, where  $(\SS_{k}^{2})$ expresses that $\SS_{k}^{2}$ exists.  
We note that the operators $\nu_{n}$ from \cite{boekskeopendoen}*{p.\ 129} are essentially $\SS_{n}^{2}$ strengthened to return a witness (if existant) to the $\Sigma_{n}^{1}$-formula at hand.  

\smallskip

\noindent
Finally, second-order arithmetic $\Z_{2}$ readily follows from $\cup_{k}\SIXK$, or from:
\be\tag{$\exists^{3}$}
(\exists E^{3}\leq_{3}1)(\forall Y^{2})\big[  (\exists f^{1})(Y(f)=0)\asa E(Y)=0  \big], 
\ee
and we therefore define $\Z_{2}^{\Omega}\equiv \RCAo+(\exists^{3})$ and $\Z_{2}^\omega\equiv \cup_{k}\SIXK$, which are conservative over $\Z_{2}$ by \cite{hunterphd}*{Cor.\ 2.6}. 
Despite this close connection, $\Z_{2}^{\omega}$ and $\Z_{2}^{\Omega}$ can behave quite differently, as discussed in Remark \ref{proke}.
The functional from $(\exists^{3})$ is also called `$\exists^{3}$', and we use the same convention for other functionals.

\subsubsection{Some basic definitions}\label{prelim2}
We introduce the higher-order definitions of `set' and `countable', as can be found in e.g.\ \cite{dagsamX, dagsamVII, dagsamXI}. 

\smallskip

First of all, open sets are represented in second-order RM as countable unions of basic open sets (\cite{simpson2}*{II.5.6}), and we refer to such sets as `RM-open'.
By \cite{simpson2}*{II.7.1}, one can effectively convert between RM-open sets and (RM-codes for) continuous characteristic functions.
Thus, a natural extension of the notion of `open set' is to allow \emph{arbitrary} (possibly discontinuous) characteristic functions, as is done in e.g.\ \cite{dagsamVII, dagsamX}.  
To make sure (basic) RM-open sets have characteristic functions, we shall always assume $\ACAo$ when necessary.
\bdefi[Subsets of $\R$]\label{openset}
We let $Y: \R \di \{0,1\}$ represent subsets of $\R$ as follows: we write `$x \in Y$' for `$Y(x)=1$'.
\edefi
The notion of `subset of $2^{\N}$ or $\N^{\N}$' now has an obvious definition. 
Having introduced our notion of set, we now turn to countable sets.
\bdefi[Enumerable sets of reals]\label{eni}
A set $A\subset \R$ is \emph{enumerable} if there exists a sequence $(x_{n})_{n\in \N}$ such that $(\forall x\in \R)(x\in A\asa (\exists n\in \N)(x=_{\R}x_{n}))$.  
\edefi
This definition reflects the RM-notion of `countable set' from \cite{simpson2}*{V.4.2}.  Note that given Feferman's $\mu^{2}$, we can remove all elements from a sequence of reals $(x_{n})_{n\in \N}$ that are not in a given set $A\subset \R$.  

\smallskip
\noindent
The definition of `countable set of reals' is now as follows in $\RCAo$, while the associated definitions for Baire space are obvious. 
\bdefi[Countable subset of $\R$]\label{standard}~
A set $A\subset \R$ is \emph{countable} if there exists $Y:\R\di \N$ such that $(\forall x, y\in A)(Y(x)=_{0}Y(y)\di x=_{\R}y)$. 
The functional $Y$ is called \emph{injective} on $A$ or \emph{an injection} on $A$.
If $Y:\R\di \N$ is also \emph{surjective}, i.e.\ $(\forall n\in \N)(\exists x\in A)(Y(x)=n)$, we call $A$ \emph{strongly countable}.
The functional $Y$ is then called \emph{bijective} on $A$ or \emph{a bijection} on $A$.
\edefi
The first part of Definition \ref{standard} is from Kunen's set theory textbook (\cite{kunen}*{p.~63}) and the second part is taken from Hrbacek-Jech's set theory textbook \cite{hrbacekjech} (where the term `countable' is used instead of `strongly countable').  According to Veldman (\cite{veldje2}*{p.\ 292}), Brouwer studied set theory based on injections.
Hereafter, `strongly countable' and `countable' shall exclusively refer to Definition~\ref{standard}.  

\smallskip

Finally, note that the principles $\NIN$ and $\NBI$ from Section \ref{intro} have now been defined. 
We have previously studied the RM of $\cocode_{i}$ for $i=0,1$ in \cite{dagsamXI, dagsamX}, where the index $i=0$ expresses that a countable set in the unit interval can be enumerated (for $i=1$, we restrict to strongly countable sets).

\section{Main results}\label{main}
We establish the results sketched in Section \ref{krum}.
We generally assume $(\exists^{2})$ from Section \ref{prelim1} to avoid the technical details involved in the representation of sets and real numbers.  
Given that $\NIN$ cannot be proved in $\Z_{2}^{\omega}$ by Remark \ref{proke}, this seems like a weak assumption.  

\subsection{Basic robustness results}\label{basic}
In this section, we show that $\NIN$, $\NBI$, and related principles are relatively robust when it comes to the domain of the mappings therein.  

\smallskip

First of all, let $\NIN^{\X}$ express that there is no injection $Y:\X\di \N$, for $\X$ equal to either the reals $\R$, Cantor space $2^{\N}$ (also denoted as $C$), or Baire space $\N^{\N}$.  
\begin{thm}\label{nintsel}
The system $\ACAo$ proves $\NIN\asa \NIN^{C}\asa \NIN^{\N^{\N}}\asa \NIN^{\R}$.  
\end{thm}
\begin{proof}
First of all, $\NIN\di \NIN^{\R}$ and $\NIN^{C}\di \NIN^{\N^{\N}}$ are trivial, while $\NIN^{\R}\di \NIN$ follows by considering the injection $\frac{1}{2}(1+\frac{x}{1+|x|})$ from $\R$ to $(0,1)$.

\smallskip

Secondly, assume $\NIN$ and use the usual interval-halving technique (using $\exists^{2}$) to obtain $\eta:[0,1]\di 2^{\N}$ such that $\eta(x)$ is the binary representation of $x\in [0,1]$, choosing a tail of zeros in the non-unique case.  
Fix $Y:2^{\N}\di \N$ and define $Z:[0,1]\di \N$ as $Z(x):= Y(\eta(x))$, which satisfies the axiom of extensionality\footnote{Functions $F:\R\di \R$ are represented by $\Phi:\N^{\N}\di \N^{\N}$ mapping equal reals to equal reals, i.e.\ extensionality as in $(\forall x , y\in \R)(x=_{\R}y\di \Phi(x)=_{\R}\Phi(y))$ (see \cite{kohlenbach2}*{p.\ 289}).} on $\R$ by definition.  
By $\NIN$, there are $x, y\in[0,1]$ with $x\ne_{\R}y$ and $Z(x)=Z(y)$.  Clearly, $\eta(x)\ne_{1} \eta(y)$ and $Y(\eta(x))=Y(\eta(y))$, and $\NIN^{C}$ follows.  

\smallskip

Thirdly, assume $\NIN^{C}$, fix $Z:[0,1]\di \N$ and let $(q_{n})_{n\in \N}$ be a list of all rational numbers with non-unique binary representation.  
Define $Y:2^{\N}\di \N$ as follows:  $Y(f):= 3 Z(\r(f)) $ in case $\r(f):=\sum_{n=0}^{\infty}\frac{f(n)}{2^{n+1}}$ has a unique binary representation, $Y(f):=3n+1$ in case $\r(f)=q_{n}$ and $f$ has a tail of zeros, and $Y(f)=3n+2$ in case $\r(f)=q_{n}$ and $f$ has a tail of ones. 
By $\NIN^{C}$, there are $f, g\in 2^{\N}$ such that $f\ne_{1} g$ and $Y(f)=Y(g)$.  Clearly, this is only possible in the first case of the definition of $Z$, i.e.\ we have $Y(f)=3Z(\r(f))=3Z(\r(g))=Y(g)$.  
Since also $\r(f)\ne_{\R}\r(g)$, $\NIN$ follows and we obtain $\NIN\asa \NIN^{C}$.

\smallskip

Finally, let $Y:2^{\N}\di \N$ be an injection.  
For $f\in \N^{\N}$, define its graph $X_{f}:=\{(n, f(n)):n \in \N\} $ in $\N^{2}$ and code the latter as a binary sequence $\tilde{X_{f}}$.   Note that $f(n):= (\mu m)[(n,m )\in X_{f}]$ recovers the function $f$ from its graph $X_{f}$.
Modulo this coding, define $Z:\N^{\N}\di \N$ as $Z(f):=Y(\tilde{X_{f}})$.  By the assumption on $Y$, $Z(f)=_{0}Z(g)$ for $f, g\in \N^{\N}$ implies $\tilde{X_{f}}=_{1}\tilde{X_{g}}$, which implies $f=_{1}g$, by the definition of $X_{f}$.  
Hence, $\neg\NIN^{C}\di \neg\NIN^{\N^{\N}}$, and we are done.    
\qed
\end{proof}
Similarly, $\cocode_{0}^{\X}$ is the statement that any countable subset of $\X$ can be enumerated, while $\cocode_{1}^{\X}$ is the restriction to strongly countable sets.  
\begin{thm}[$\ACAo$]\label{restofthebest}
For $i=0,1$, we have $\cocode_{i}\asa \cocode_{i}^{\R}\asa \cocode_{i}^{C}$.  
\end{thm}
\begin{proof}
The implication $\cocode_{i}^{\R}\di \cocode_{i}$ is trivial while the (rescaled) arctangent function is a bijection from $\R$ to $(0,1)$, which readily yields the reversal. 

\smallskip

Now assume $\cocode_{0}^{C}$ and let $Z:[0,1]\di \N$ be injective on $A\subset [0,1]$.     
The functional $Y:2^{\N}\di \N$ defined by $Y(f):=Z(\r(f))$ is clearly injective on $B:=\{\eta(x):x\in A\}$ where $\eta$ is as in the proof of Theorem \ref{nintsel}.  
Let $(f_{n})_{n\in \N}$ be a list of all elements in $B$ and note that $(\r(f_{n}))_{n\in \N}$ is a list of all elements in $A$, i.e.\ $\cocode_{0}$ follows. 
Note that if $Z$ is bijective on $A$, then $Y$ is bijective on $B$ by definition, i.e.\ $\cocode_{1}^{C}\di \cocode_{1}$.  

\smallskip

Next, assume $\cocode_{0}$, let $Y:2^{\N}\di \N$ be injective on $A\subset 2^{\N}$, and define $Z(x):=Y(\eta(x))$.  
Then $Z:[0,1]\di \N$ witnesses that $B=\{ \r(f): f\in A\}$ is countable, and let $(x_{n})_{n\in \N}$ be an enumeration of $B$.
This list is readily converted to a list of all elements in $A$ via $\eta$ and by noting that $\mu^{2}$ can list all $f\in A$ such that $\r(f)$ has a non-unique binary representation; we thus have $\cocode_{0}^{C}$.

\smallskip

We now prove $\cocode_{1}^{\R}\di \cocode_{1}^{C}$.
Let $Y:2^{\N}\di \N$ be bijective on $A\subset 2^{\N}$ and let $(f_{n})_{n\in \N}$ be the list of all $f\in A$ such that $\r(f)$ has a non-unique binary representation.  
Now define $D\subset \R$ as: $x\in D$ if either of the following holds: 
\begin{itemize}
\item $x\in [0,1]$, $x$ has a unique binary representation, and $\eta(x)\in A$, 
\item there is $n\in \N$ with $x\in (n,+1, n+2] $ and $x-(n+1)=_{\R}\r(f_{n})$.
\end{itemize}
Define $W:\R\di \N$ as $W(x):= Y(\eta(x))$ if $x\in [0,1]$ and $W(x):= Y(f_{n})$ in case $|x|\in (n+1, n+2]$ as in the second case of the definition of $D$.  
Then $W$ is a bijection on $D$ since $Y$ is a bijection on $A$.  The list provided by $\cocode_{1}^{\R}$ for $D$ now readily yields the list required for $A$ as in $\cocode_{1}^{C}$.  
\qed\end{proof}
\noindent
Finally, $\NBI^{\X}$ is the statement that there is no bijection from $\X$ to $\N$, where $\X$ is e.g.\ $\R$ or $\N^{\N}$. 
We have the following theorem. 
\begin{thm}\label{hisproof}
The system $\ACAo$ proves $\NBI\asa \NBI^{\R}$ and $\NBI\di \NBI^{\N^{\N}}$.
\end{thm}
\begin{proof}
The implication $\NBI\di \NBI^{\R}$ is immediate as the (rescaled) tangent function provides a bijection from $(0,1)$ to $\R$.  
The inverse of tangent, called \emph{arctangent}, yields a bijection in the other direction (also with rescaling), i.e.\ the first equivalence 
is immediate, as well as $\NBI\asa \NBI^{\R_{\geq_{0}}}$.
We now define a (continuous) bijection from $\N^{\N}$ to $\R_{\geq0}$ based on \emph{continued fractions}.  
Intuitively, a sequence $(a_{n})_{n\in \N}$ of natural numbers is mapped to the real $x\in \R_{\geq0}$ via the following (generalised) continued fraction: 
\be\label{trix}\tag{CF}
x=a_{0}+{\cfrac {1}{1+{\cfrac {1}{a_{1}+{\cfrac {1}{1+{\cfrac {1}{a_{2}+\ddots \,}}}}}}}}
\ee
The real $x\in \R_{\geq0}$ in \eqref{trix} exists in $\ACAo$ in the sense that there is an explicit function $F:(\N^{\N}\times n)\di \Q$ such that $x=_{\R}\lim_{n\di \infty }F(f)(n)$, where $F(f)(n)\in \Q$ is essentially the continued fraction in \eqref{trix} `broken off' after encountering $a_{n}$.
The definition of $F$ can be be found in e.g.\ \cite{loppper}*{Ch.1, p.\ 7-9}.  
One readily shows that the mapping defined by \eqref{trix} is a bijection from $\N^{\N}$ to $\R_{\geq 0}$ in $\ACAo$.  
\qed
\end{proof}
We could prove similar results for \emph{a countable set in the unit interval has measure\footnote{For $A\subset \R$, let `$A$ has measure zero' mean that for any $\eps>0$, there is a sequence of closed intervals $\big(I_{n}\big)_{n\in \N}$ covering $A$ and such that $\eps >\sum^{\infty}_{n=0 }|J_{n}|$ for $J_{0}:= I_{0}$ and $J_{i+1}:= I_{i+1}\setminus \cup_{j\leq i}I_{j}$.  This follows from the usual definition as used in mathematics.\label{clukker}} zero}, which is intermediate between $\cocode_{0}$ and $\NIN$, which is shown in \cite{dagsamX} as an illustration how weak $\NIN$ is.  Nonetheless, we have the following result.  
\begin{thm}[$\ACAo$]\label{haha}
A countable set $A\subset [0,1]$ has \emph{weak\footnote{For $A\subset \R$, let `$A$ has weak measure zero' mean that for any $\eps>0$, there is a sequence $(\eps_{n})_{n\in \N}$, a set $B$ of closed intervals, and $Z:\R^{2}\di \N$ injective on $B$, such that 
$(\forall a\in A)(\exists (b, c)\in B)(a\in (b, c) )$ and $(\forall (b,c)\in B, \forall n\in \N)( Z((b, c))=n\di |b-c|\leq \eps_{n}  )$ and $\eps\geq\sum_{n=0}^{\infty}\eps_{n}$.  Given $\cocode_{0}$, this is the same as `measure zero'.\label{clukker2}} measure zero}.
\end{thm}
\begin{proof}
Fix $A\subset [0,1]$ and $Y:[0,1]\di \N$ injective on $A$.  For $\eps>0$, define $\eps_{n}:= \frac{\eps}{2^{n+1}}$, $B:= \{ (a, b)\in \R^{2}: \frac{a+b}{2}\in A \wedge |b-a|={2^{-Y(\frac{a+b}{2})}}  \} $, and $Z((a,b)):=Y(\frac{a+b}{2})$.
Clearly, this shows that $A$ has weak measure zero, as required. \qed
\end{proof}
We say that a property holds \emph{weakly almost everywhere} (wae) in case it holds outside a set of weak measure zero as in Footnote \ref{clukker2}.

\smallskip

We finish this section with a conceptual remark regarding our base theory.
\begin{rem}\label{LEM}\rm
We have used $\ACAo$ as the base theory for the above results, since our notion of `set-as-characteristic function' as in Definition \ref{openset} is poorly behaved in the absence of $(\exists^{2})$.
One \emph{can} obtain equivalences over $\RCAo$, and let us establish $\NIN^{\N^{\N}}\di \NIN^{C}$ over $\RCAo$ as an example via the following steps.
\begin{itemize}
\item Fix any $Y:2^{\N}\di \N$, which may or may not be continuous.
\item In case $Y$ is \emph{continuous}, it is immediate that $Y(00\dots)=Y(00\dots00*11\dots)$ for enough instances of $0$ on the right.  
\item In case $Y$ is \emph{discontinuous}, use the results in \cite{kohlenbach2}*{\S3} to derive $(\exists^{2})$ over $\RCAo$.  We can now use the proof of Theorem \ref{nintsel} in $\ACAo$.
\end{itemize}
The above proof of course heavily relies on the law of excluded middle.  
\end{rem}

\subsection{Advanced robustness results}\label{adv}
In this section, we show that $\NIN$ is equivalent to various restrictions involving notions from mainstream mathematics, like semi-continuity and bounded variation; we first introduce the latter. 

\smallskip

First of all, an important weak continuity notion is \emph{semi-continuity}, introduced by Baire in \cite{beren2} around 1899.  
By \cite{beren2}*{\S84, p.\ 94-95}, the notion of quasi-continuity goes back to Volterra; any cliquish function is the sum of two quasi-continuous functions.  
Moreover, while the limits in the following definition may not exist in $\RCAo$, the associated inequalities always make sense.  
\bdefi[Weak continuity]\label{WCloset}~
\begin{itemize}
\item $f:\R\di \R$ is \emph{upper semi-continuous} if for all $x_{0}\in \R$, $f(x_{0})\geq_{\R}\lim\sup_{x\di x_{0}} f(x)$.
\item $f:\R\di \R$ is \emph{lower semi-continuous} if for all $x_{0}\in \R$, $f(x_{0})\leq_{\R}\lim\inf_{x\di x_{0}} f(x)$.
\item $ f:X\rightarrow \mathbb {R} $ is \emph{quasi-continuous} (resp.\ \emph{cliquish}) at $x\in X$ if for any $ \epsilon > 0$ and any open neighbourhood $U$ of $x$, there is a non-empty open ball ${ G\subset U}$ with $(\forall y\in G) (|f(x)-f(y)|<\eps)$ (resp.\ $(\forall y, z\in G) (|f(z)-f(y)|<\eps)$).
\end{itemize}
\edefi
Secondly, Jordan introduces the notion of \emph{bounded variation} in \cite{jordel} around 1881, also studied in second-order RM (\cite{nieyo, kreupel}).
Moreover, Jordan proves in \cite{jordel3}*{\S105} that functions of bounded variation are exactly those for which the notion of `length of the graph' makes sense; the latter boast\footnote{The notion of arc length was studied for discontinuous regulated functions in 1884  (\cite{scheeffer}*{\S1-2}), where it is also claimed to be essentially equivalent to Duhamel's 1866 approach from \cite{duhamel}*{Ch.\ VI}.  Around 1833, Dirksen, the PhD supervisor of Jacobi and Heine, provides a definition of arc length that is (very) similar to the modern one (see \cite{dirksen}*{\S2, p.\ 128}), but with some conceptual problems as discussed in \cite{coolitman}*{\S3}.
} an even `earlier' history.   What is more, Lakatos in \cite{laktose}*{p.\ 148} claims that Jordan did not invent or introduce the notion of bounded variation in \cite{jordel}, but rather discovered it in Dirichlet's 1829 paper \cite{didi3}.
\bdefi[Bounded variation]\label{varvar}
Any $f:[a,b]\di \R$ \emph{has bounded variation} on $[a,b]$ if there is $k_{0}\in \N$ such that $k_{0}\geq \sum_{i=0}^{n} |f(x_{i})-f(x_{i+1})|$ 
for any partition $x_{0}=a <x_{1}< \dots< x_{n-1}<x_{n}=b  $.
\edefi
Functions of bounded variation have only got countably many points of discontinuity (see e.g.\ \cite{voordedorst}*{Ch.\ 1}); Dag Normann and the author study 
this property in higher-order computability theory in \cite{dagsamXII}.  In the latter, we also study regulated functions (called `regular' in \cite{voordedorst}), defined as follows (say in $\ACAo$). 
\bdefi[Regulated function] 
A function $f:[0,1]\di \R$ is \emph{regulated} if for every $x_{0}\in [0,1]$, the `left' and `right' limit $f(x_{0}-)=\lim_{x\di x_{0}-}f(x)$ and $f(x_{0}+)=\lim_{x\di x_{0}+}f(x)$ exist.  
\edefi
Thirdly, Borel functions are defined in Definition \ref{defzie}; the usual definition of Borel set makes sense in $\ACAo$, where $(\exists^{2})$ is used to define countable unions.
\bdefi[Borel function]\label{defzie}
Any $f:[0,1]\di \R$ is a Borel function in case $f^{-1}((a, +\infty)):=\{x\in [0,1]:f(x)>a\}$ is a Borel set for any $a\in \R$.
\edefi
Fourth, recall the induction axiom $\IND_{0}$ from Section \ref{prelim2}.  Let $\Y$ be any property such that `$f:[0,1]\di \R$ satisfies $\Y$' follows from `$f$ has bounded variation on $[0,1]$' and where this implication can be established over (say) $\ACAo$.
\begin{thm}[$\ACAo+\IND_{0}$]\label{inyourface}
The following are equivalent to $\NIN$:
\begin{itemize}
\item $\NIN_{\bv}$: there is no injection from $[0,1]$ to $\Q$ that has bounded variation, 
\item $\NIN_{\textup{\textsf{Y}}}$: there is no injection from $[0,1]$ to $\Q$ that has property $\Y$,
\item $\NIN_{\textup{\textsf{Riemann}}}$: there is no injection from $[0,1]$ to $\Q$ that is Riemann integrable,
\item $\NIN_{\Borel}$: there is no Borel function that is an injection from $[0,1]$ to $\Q$,
\item $\NIN_{\textsf{\textup{reg}}}$: there is no injection from $[0,1]$ to $\Q$ that is regulated, 
\item $\NIN_{\cliq}$: there is no injection from $[0,1]$ to $\Q$ that is cliquish, 
\item $\NIN_{\semi}$: there is no upper semi-continuous injection from $[0,1]$ to $\Q$,
 \item $\NIN_{\semi}'$: there is no lower semi-continuous injection from $[0,1]$ to $\Q$. 
 \end{itemize}
 Only the implications involving the final five items require the use of $\IND_{0}$. 
\end{thm}
\begin{proof}
As there is an injection from $\Q$ to $\N$ in $\RCA_{0}$, we only need to prove that $\NIN_{\bv}\di \NIN$ over $\ACAo$ for the first equivalence.
To this end, let $Y:[0,1]\di \N$ be an injection and define $W:[0,1]\di \Q$ by $W(x):=\frac{1}{2^{Y(x)+1}}$.  
Then $W$ has bounded variation with upper bound $2$.  Indeed, since $Y$ is an injection on $[0,1]$, any sum $\sum_{i=0}^{n} |W(x_{n})-W(x_{n+1})|$ is at most $\sum_{i=0}^{n} \frac{1}{2^{i+1}}$.
By $\NIN_{\bv}$, there are $x, y\in [0,1]$ with $x\ne_{\R} y$ and $W(x)=_{\Q} W(y)$.
This implies the contradiction $Y(x)=_{0} Y(y)$, and $\NIN\asa \NIN_{\bv}$ follows.  
For $\NIN_{\textup{\textsf{Riemann}}}\di \NIN$, the function $W$ is Riemann integrable following the $\eps$-$\delta$-definition. 
Indeed, fix $\eps_{0}>0$ and find $k_{0}\in \N$ such that $\frac{1}{2^{k_{0}}}<\eps_{0}$.  Since $Y$ is an injection, if $P$ is a partition of $[0,1]$ consisting of $|P|$-many points and with mesh $\|P\|\leq \frac{1}{2^{k_{0}}}$, it is immediate that the Riemann sum $S(W, P)$ is smaller than $\frac{1}{2^{k_{0}}}\sum_{n=0}^{|P|}\frac{1}{2^{i+1}}$, which is at most $\frac{1}{2^{k_{0}}}$. 

\smallskip

For the implication $\NIN_{\semi}\di \NIN$, consider the same $W:[0,1]\di \R$ and note that $[\lim\sup_{x\di x_{0}}W(x)]=_{\R}0<_{\R}W(x_{0})$ for any $x_{0}\in [0,1]$ in case $Y:[0,1]\di \N$ is an injection.  Hence, $W(x)$ is upper semi-continuous and $Z(x):=1-W(x)$ is similarly \emph{lower} semi-continuous, since $[\lim\inf_{x\di x_{0}}Z(x)]=_{\R}1>_{\R}Z(x_{0})$ for any $x_{0}\in [0,1]$.  The finite sequences provided by $\INDY$ seem essential to establish these semi-continuity claims.    
One proves $\NIN_{\cliq}\di \NIN$ in the same way, namely using $\IND_{0}$ to exclude the finitely many `too large' function values.  
For the implication $\NIN_{\Borel}\di \NIN$, note that for an injection $Y:[0,1]\di \N$ the above function $W(x)$ is Borel as $W^{-1}\big ((a, +\infty) \big)$ for any $a\in \R$ is either finite or $[0,1]$, and that these are Borel sets is immediate in $\ACAo+\IND_{0}$.
For the implication $\NIN_{\textsf{reg}}\di \NIN$, consider the same $W:[0,1]\di \R$ and note that $W(0+)=W(1-)=W(x+)=W(x-)=0$ for $x\in (0,1)$ in the same way as for the semi-continuity of $W$.
Thus, $W$ is regulated and we are done. 
\qed
\end{proof}
As noted above, a function has bounded variation iff it has finite arc length.  The proof of this equivalence (\cite{voordedorst}*{Prop.\ 3.28}) goes through in $\RCAo$, i.e.\ we may replace `bounded variation' by `finite arc length' in the previous theorem. 

\smallskip

Fifth, we say that a function has \emph{total variation equal to $a\in \R$} in case the supremum over all partitions of $ \sum_{i=0}^{n} |f(x_{i})-f(x_{i+1})|$ in Def.\	 \ref{varvar} equals $a$.
\begin{cor}[$\ACAo+\IND_{0}$]\label{inyourfacecor}
The following are equivalent to $\NBI$:
\begin{itemize}
\item $\NBI_{\textup{\textsf{Riemann}}}$: there is no bijection from $[0,1]$ to $\Q$ that is Riemann integrable,
\item $\NBI_{\bv}$: there is no injection from $[0,1]$ to $\Q$ that has total variation $1$, 
\item $\NBI_{\Borel}$: there is no Borel function that is a bijection from $[0,1]$ to $\Q$,
\item $\NBI_{\cliq}$: there is no bijection from $[0,1]$ to $\Q$ that is cliquish, 
\item $\NBI_{\semi}$: there is no upper semi-continuous bijection from $[0,1]$ to $\Q$,
 \item $\NBI_{\semi}'$: there is no lower semi-continuous bijection from $[0,1]$ to $\Q$. 
 \end{itemize}
Only the implications involving the final four items require the use of $\IND_{0}$. 
\end{cor}
\begin{proof}
For the first equivalence, $W:[0,1]\di \R$ from the proof has total variation \emph{exactly} $1$ in case $Y$ is also surjective.  The other equivalences are now immediate by the proof of the theorem.
\qed
\end{proof}
As an intermediate conclusion, one readily proves that there are no \emph{continuous} injections from $\R$ to $\Q$ (say over $\ACAo$).  However, Theorem~\ref{inyourface} and Corollary~\ref{inyourfacecor} show that 
admitting countably many points of discontinuity, one obtains principles that are extremely hard to prove following Remark \ref{proke}. 

\smallskip

Finally, one can greatly generalise Theorem \ref{inyourface} based on Remark \ref{essenti}.  
Indeed, there are many spaces intermediate between bounded variation and regulated, each of which yields a natural and equivalent restriction of $\NIN$.  
\begin{rem}[Intermediate spaces]\label{essenti}\rm 
The following spaces are intermediate between bounded variation and regulated; all details may be found in \cite{voordedorst}.  
Wiener spaces from mathematical physics are based on \emph{$p$-variation}, which amounts to replacing `$ |f(x_{i})-f(x_{i+1})|$' by `$ |f(x_{i})-f(x_{i+1})|^{p}$' in the definition of variation. 
Young generalises this to \emph{$\phi$-variation} which instead involves $\phi( |f(x_{i})-f(x_{i+1})|)$ for so-called Young functions $\phi$, yielding the Wiener-Young spaces.  
Perhaps a simpler construct is the Waterman variation, which involves $ \lambda_{i}|f(x_{i})-f(x_{i+1})|$ and where $(\lambda_{n})_{n\in \N}$ is a sequence of reals with nice properties; in contrast to bounded variation, any continuous function is included in the Waterman space (\cite{voordedorst}*{Prop.\ 2.23}).  Combining ideas from the above, the \emph{Schramm variation} involves $\phi_{i}( |f(x_{i})-f(x_{i+1})|)$ for a sequence $(\phi_{n})_{n\in \N}$ of well-behaved `gauge' functions.  
As to generality, the union (resp.\ intersection) of all Schramm spaces yields the space of regulated (resp.\ bounded variation) functions, while all other aforementioned spaces are Schramm spaces (\cite{voordedorst}*{Prop.\ 2.43 and 2.46}).
In contrast to bounded variation and the Jordan decomposition theorem, these generalised notions of variation have no known `nice' decomposition theorem.  The notion of \emph{Korenblum variation} does have such a theorem (see \cite{voordedorst}*{Prop.\ 2.68}) and involves a distortion function acting on the \emph{partition}, not on the function values.  
\end{rem}

\subsection{Connections to mainstream mathematics}\label{knew}
We establish the connection between $\NIN$ and two theorems from mainstream mathematics, namely \emph{Cousin's lemma} and \emph{Jordan's decomposition theorem}.

\smallskip

First of all, our results have significant implications for the RM of Cousin's lemma.  
Indeed, as shown in \cites{dagsamIII}, $\Z_{2}^{\omega}$ cannot prove Cousin's lemma as follows:
\be\tag{$\HBU$}
(\forall \Psi:\R\di \R^{+})(\exists  y_{0}, \dots, y_{k}\in [0,1])([0,1]\subset \cup_{i\leq k} B(y_{i}, \Psi(y_{i}))),
\ee
which expresses that the \emph{canonical covering} $\cup_{x\in [0,1]}B(x, \Psi(x))$ has a finite sub-covering, namely given by $y_{0}, \dots, y_{k}\in [0,1]$.  
In \cite{basket2}, it is shown that $\HBU$ formulated using \emph{second-order codes} for Borel functions is provable in $\ATR_{0}$ plus some induction.  
We now show that this result from \cite{basket2} is entirely due to the presence of second-order codes.  Indeed, by Theorem \ref{genewegneffen}, the restriction of $\HBU$ to Borel functions still implies $\NIN$, which is not provable in $\Z_{2}^{\omega}$ by Remark \ref{proke}.
To this end, let $\HBU_{\semi}$ (resp.\ $\HBU_{\Borel}$) be $\HBU$ restricted to $\Psi:[0,1]\di \R^{+}$ that are upper semi-continuous (resp.\ Borel) as in Definition \ref{WCloset} (resp.\ Def.\ \ref{defzie}). 
\begin{thm}[$\ACAo+\IND_{0}$]\label{genewegneffen}
$\NIN$ follows from $\HBU_{\textsf{\textup{\semi}}}$ and from $ \HBU_{\Borel}$; extra induction is only needed in the first case.  
\end{thm}
\begin{proof}
Let $Y:[0,1]\di \N$ be an injection and consider $\Psi(x):=\frac{1}{2^{Y(x)+3}}$, which is upper semi-continuous and Borel by the proof of Theorem \ref{inyourface}.  
Now consider the uncountable covering $\cup_{x\in [0,1]}B(x, \frac{1}{2^{Y(x)+3}})$ of $[0,1]$.  
Since $Y$ is an injection, we have $\sum_{i\leq k}|B(x_{i}, \frac{1}{2^{Y(x_{i})+3}})|\leq \sum_{i\leq k} \frac{1}{2^{i+2}} \leq \frac{1}{2}$ for any finite sequence $x_{0}, \dots, x_{k}$ of distinct reals in $[0,1]$.    
In this light, $\HBU_{\semi}$ and $\HBU_{\Borel}$ are false.  
We note that the required basic measure theory (for finite sequences of intervals) can be developed in $\RCA_{0}$ (\cite{simpson2}*{X.1}). 
\qed
\end{proof}
We now show that we can replace `Borel' by `Baire class 2' in Theorem \ref{genewegneffen}, assuming the right (equivalent) definition.  
Now, \emph{Baire classes} go back to Baire's 1899 dissertation (\cite{beren2}) and a function is `Baire class $0$' if it is continuous and `Baire class $n+1$' if it is the pointwise limit of Baire class $n$ functions.  
Baire's \emph{characterisation theorem} (\cite{beren}*{p.\ 127}) expresses that a function is Baire class $1$ iff there is a point of continuity of the induced function on each perfect set.

\smallskip

Now let $\textsf{B2}$ be the class of all $g:[0,1]\di \R$ such that $g=\lim_{n\di \infty }g_{n}$ on $[0,1]$ and where for all $n\in\N$ and perfect $P\subset [0,1]$, the restriction ${g_{n}}_{\upharpoonright P}$ has a point of continuity on $P$. 
We have the following corollary.  
\begin{cor}[$\ACAo+\IND_{0}$]\label{dokcor}
We have $\HBU_{\textsf{\textup{B2}}}\di \NIN$ where the former is the restriction of $\HBU$ to $\Psi:[0,1]\di \R^{+}$ in $\textup{\textsf{B2}}$.
\end{cor}
\begin{proof}
Fix $A\subset [0,1]$ and $Y:[0,1]\di \N$ with $Y$ is injective on $A$.  
Define $\Psi:[0,1]\di \R^{+}$ as follows: $\Psi(x)$ is $\frac{1}{2^{Y(x)+5}}$ in case $x\in A$, and $1/8$ otherwise.  
Define $\Psi_{n}$ as $\Psi$ with the condition `$Y(x)\leq n+5$' in the first case. 
Clearly $\Psi=\lim_{n\di \infty}\Psi$ and $\Psi\in \textsf{B2}$, as $\Psi_{n}$ only has at most $n+5$ points of discontinuity (the set of which is not perfect in $\ACAo+\IND_{0}$).
For a finite sub-covering $x_{0}, \dots, x_{k}\in [0,1]$ of $\cup_{x\in [0,1]}B(x, \Psi(x))$, there must be $j\leq k$, with $x_{j}\not \in A$.  
Indeed, the measure of $\cup_{i\leq k} B(x_{i}, \Psi(x_{i}))$ is otherwise below $\sum_{n=0}^{k}\frac{1}{2^{i+5}}<1$, a contradiction as the required basic measure theory can be developed in $\RCA_{0}$ (\cite{simpson2}*{X.1}). \qed
\end{proof}
Secondly, Jordan proves the following fundamental theorem about functions of bounded variation around 1881 in \cite{jordel}.
\begin{thm}[Jordan decomposition theorem]\label{drd}
Any $f : [0, 1] \di \R$ of bounded variation is the difference of two non-decreasing functions $g, h:[0,1]\di \R$.
\end{thm} 
Formulated using second-order codes, Theorem \ref{drd} is provable in $\ACA_{0}$ (see \cite{nieyo, kreupel}); we now show that the third-order version
is \emph{hard to prove} as in Remark \ref{proke}.
\begin{thm}[$\ACAo$] Each item implies the one below it.  
\begin{itemize}
\item The Jordan decomposition theorem for the unit interval. 
\item $\HBU_{\bv}$, i.e.\ $\HBU$ restricted to $\Psi:[0,1]\di \R^{+}$ of bounded variation.  
\item $\NIN$: there is no injection from $[0,1]$ to $\N$. 
\end{itemize}
Assuming $\IND_{0}$, we may replace the principle $\HBU_{\bv}$ by the following one:
\begin{itemize}
\item For $f:[0,1]\di \R$ of bounded variation, there is $x\in [0,1]$ such that $f$ is continuous \(or: quasi-continuous\) at $x$. 
\end{itemize}
\end{thm}
\begin{proof}
The poeints of discontinuity of a non-decreasing function can be enumerated in $\ACAo$ by \cite{dagsamXII}*{Lemma 3.3}.
Now assume the Jordan decomposition theorem and fix some $\Psi:[0,1]\di \R^{+}$ of bounded variation.  
If $(x_{n})_{n\in \N}$ enumerates all the points of discontinuity of $\Psi$, then the following also covers $[0,1]$.
\[
\cup_{q\in \Q\cap [0,1]} B(q, \Psi(q))\bigcup \cup_{n\in \N}B(x_{n}, \Psi(x_{n})).
\]
The second-order Heine-Borel theorem (provable in $\WKL_{0}$ by \cite{simpson2}*{IV.1}) now yields a finite sub-covering, and $\HBU_{\bv}$ follows.  
Now assume the latter and suppose $Y:[0,1]\di \N$ is an injection.  Define $\Psi:[0,1]\di \N$ as $\Psi(x):=\frac{1}{2^{Y(x)+3}}$.  
As in the proof of Corollary \ref{dokcor}, any finite sub-covering of $\cup_{x\in [0,1]}B(x, \Psi(x))$ must have measure at most $1/2$, a contradiction; $\NIN$ follows and the first part is done.

\smallskip

For the second part of the theorem, we use the first part of the proof, namely that for $f:[0,1]\di \R$ of bounded variation, the points of discontinuity can be enumerated, say by $(x_{n})_{n\in \N}$.
By \cite{simpson2}*{II.4.9}, the unit interval cannot be enumerated, i.e.\ there is $y\in [0,1]$ such that $(\forall n\in \N)(x_{n}\ne y)$.  By definition, $f$ is continuous at $y$.  
For the final implication, consider $\Psi:[0,1]\di \R^{+}$ from the first part of the proof.  The function $\Psi$ is everywhere discontinuous in case $Y$ is an injection; one seems to need $\IND_{0}$ to prove this.  
Similarly, $\Psi$ is not quasi-continuous at any $x\in [0,1]$, and we are done. 
\qed
\end{proof}

In conclusion, basic third-order theorems like Cousin's lemma and Jordan's decomposition theorem are `hard to prove' in terms of conventional comprehension following Remark \ref{proke}.  
Rather than measuring logical strength in terms of the one-dimensional scale provided by conventional comprehension, we propose an alternative \emph{two-dimensional} scale, where the first dimension is based on conventional comprehension and the second dimension is based on the \emph{neighbourhood function principle} $\NFP$ (see e.g.\ \cite{troeleke1}).  Thus, higher-order RM should seek out the minimal axioms needed to prove a given theorem of third-order arithmetic \textbf{and} these minimal axioms are in general a pair, namely a fragment of conventional comprehension and a fragment of $\NFP$.  This two-dimensional picture already exists in set theory where one studies which fragment of $\ZF$ and which fragments of $\AC$ are needed for proving a given theorem of $\ZFC$.  Note that $\ZF$ proves $\NFP$ as the choice functions in the latter are \emph{continuous}.  

\medskip

 \bibliographystyle{natbib}
{
\renewcommand{\clearpage}{} 
\bibliography{bibhope}

\begin{bibdiv}
\begin{biblist}

\bib{voordedorst}{book}{
      author={{Appell, J\"{u}rgen}},
      author={{Bana\'{s}, J\'{o}zef}},
      author={{Merentes, Nelson}},
       title={Bounded variation and around},
   publisher={De Gruyter, Berlin},
        date={2014},
      volume={17},
}

\bib{beren2}{article}{
      author={{Baire, Ren\'{e}}},
       title={Sur les fonctions de variables r\'eelles},
        date={1899},
     journal={Ann. di Mat.},
       pages={1\ndash 123},
}

\bib{beren}{book}{
      author={{Baire, Ren\'{e}}},
       title={Le\c {c}ons sur les fonctions discontinues},
      series={Les Grands Classiques Gauthier-Villars},
   publisher={\'{E}ditions Jacques Gabay},
        date={1995},
        note={Reprint of the 1905 original},
}

\bib{basket2}{article}{
      author={{Barrett, Jordan}},
      author={{Downey, Rodney}},
      author={{Greenberg, Noam}},
       title={Cousin's lemma in second-order arithmetic},
        date={2021},
     journal={Preprint, arxiv: \url{https://arxiv.org/abs/2105.02975}},
}

\bib{boekskeopendoen}{book}{
      author={{Buchholz, Wilfried}},
      author={{Feferman, Solomon}},
      author={{Pohlers, Wolfram}},
      author={{Sieg, Wilfried}},
       title={Iterated inductive definitions and subsystems of analysis},
      series={LNM 897},
   publisher={Springer},
        date={1981},
}

\bib{cantor1}{article}{
      author={{Cantor, Georg}},
       title={Ueber eine eigenschaft des inbegriffs aller reellen algebraischen
  zahlen},
        date={1874},
     journal={J. Reine Angew. Math.},
      volume={77},
       pages={258\ndash 262},
}

\bib{coolitman}{article}{
      author={{Coolidge, J. L.}},
       title={The lengths of curves},
        date={1953},
     journal={Amer. Math. Monthly},
      volume={60},
       pages={89\ndash 93},
}

\bib{didi3}{book}{
      author={{Dirichlet, L.\ P.\ G.}},
       title={{\"U}ber die darstellung ganz willk\"urlicher funktionen durch
  sinus- und cosinusreihen},
   publisher={Repertorium der physik, bd. 1},
        date={1837},
}

\bib{dirksen}{article}{
      author={{Dirksen, Enno}},
       title={{Ue}ber die anwendung der analysis auf die rectification der
  curven},
        date={1833},
     journal={Akademie der Wissenschaften zu Berlin},
       pages={123\ndash 168},
}

\bib{duhamel}{book}{
      author={{Dunham, J.\ M.\ C.}},
       title={Application des m\'ethodes g\'en\'erales \`a la science des
  nombres et \`a la science de l'\'etendue},
   publisher={Vol II, Gauthier-Villars},
        date={1866},
}

\bib{fried}{incollection}{
      author={{Friedman, Harvey}},
       title={Some systems of second order arithmetic and their use},
        date={1975},
       pages={235\ndash 242},
}

\bib{fried2}{article}{
      author={{Friedman, Harvey}},
       title={Systems of second order arithmetic with restricted induction, i
  \& ii \protect \(abstracts\protect \)},
        date={1976},
     journal={Journal of Symbolic Logic},
      volume={41},
       pages={557\ndash 559},
}

\bib{fomo}{article}{
      author={{Friedman, Harvey}},
       title={Remarks on reverse mathematics /1},
        date={Sept.\ 21st, 2021},
     journal={FOM mailing list},
        note={Website:
  \url{https://cs.nyu.edu/pipermail/fom/2021-September/022875.html}},
}

\bib{hrbacekjech}{book}{
      author={{Hrbacek, Karel}},
      author={{Jech, Thomas}},
       title={Introduction to set theory},
     edition={3},
      series={Monographs and Textbooks in Pure and Applied Mathematics},
   publisher={Marcel Dekker},
        date={1999},
      volume={220},
}

\bib{hunterphd}{book}{
      author={{Hunter, James}},
       title={Higher-order reverse topology},
   publisher={ProQuest LLC, Ann Arbor, MI},
        date={2008},
        note={Thesis (Ph.D.)--The University of Wisconsin - Madison},
}

\bib{jordel}{article}{
      author={{Jordan, Camille}},
       title={Sur la s\'erie de fourier},
        date={1881},
     journal={Comptes rendus de l'Acad\'emie des Sciences, Paris,
  Gauthier-Villars},
      volume={92},
       pages={228\ndash 230},
}

\bib{jordel3}{book}{
      author={{Jordan, Camille}},
   publisher={\'{E}ditions Jacques Gabay},
        date={1991},
        note={Reprint of the third (1909) edition; first edition: 1883},
}

\bib{kohlenbach2}{incollection}{
      author={{Kohlenbach, Ulrich}},
       title={Higher order reverse mathematics},
        date={2005},
       pages={281\ndash 295},
}

\bib{kreupel}{article}{
      author={{Kreuzer, Alexander P.}},
       title={Bounded variation and the strength of helly's selection theorem},
        date={2014},
     journal={Log. Methods Comput. Sci.},
      volume={10},
      number={4},
       pages={4:16, 15},
}

\bib{kunen}{book}{
      author={{Kunen, Kenneth}},
       title={Set theory},
      series={Studies in Logic},
   publisher={College Publications},
        date={2011},
      volume={34},
}

\bib{laktose}{book}{
      author={{Lakatos, Imre}},
        date={2015},
}

\bib{loppper}{book}{
      author={{Lorentzen, Lisa}},
      author={{Waadeland, Haakon}},
       title={Continued fractions with applications},
      series={Studies in Computational Mathematics},
   publisher={North-Holland},
        date={1992},
      volume={3},
}

\bib{montahue}{article}{
      author={{Montalb{\'a}n, Antonio}},
       title={Open questions in reverse mathematics},
        date={2011},
     journal={Bull. Sym. Logic},
      volume={17},
      number={3},
       pages={431\ndash 454},
}

\bib{neeman}{article}{
      author={{Neeman, Itay}},
       title={Necessary use of {$\Sigma_{1}^{1}$} induction in a reversal},
        date={2011},
     journal={J. Symbolic Logic},
      volume={76},
      number={2},
       pages={561\ndash 574},
}

\bib{nieyo}{article}{
      author={{Nies, Andr\'e}},
      author={{Triplett, Marcus A.}},
      author={{Yokoyama, Keita}},
       title={The reverse mathematics of theorems of jordan and lebesgue},
        date={2021},
     journal={The Journal of Symbolic Logic},
       pages={1\ndash 18},
}

\bib{dagsamIII}{article}{
      author={{Normann, Dag}},
      author={{Sanders, Sam}},
       title={On the mathematical and foundational significance of the
  uncountable},
        date={2019},
     journal={Journal of Mathematical Logic},
}

\bib{dagsamVII}{article}{
      author={{Normann, Dag}},
      author={{Sanders, Sam}},
       title={Open sets in reverse mathematics and computability theory},
        date={2020},
     journal={Journal of Logic and Computation},
      volume={30},
      number={8},
       pages={pp.\ 40},
}

\bib{dagsamV}{article}{
      author={{Normann, Dag}},
      author={{Sanders, Sam}},
       title={Pincherle's theorem in reverse mathematics and computability
  theory},
        date={2020},
     journal={Ann. Pure Appl. Logic},
      volume={171},
      number={5},
       pages={102788, 41},
}

\bib{dagsamXI}{article}{
      author={{Normann, Dag}},
      author={{Sanders, Sam}},
       title={On robust theorems due to {B}olzano, {W}eierstrass, and {C}antor
  in reverse mathematics},
        date={2021},
     journal={Submitted, see \url{https://arxiv.org/abs/2102.04787}},
       pages={pp.\ 30},
}

\bib{dagsamXII}{article}{
      author={{Normann, Dag}},
      author={{Sanders, Sam}},
       title={Betwixt turing and kleene},
        date={2022},
     journal={LNCS 13137, Proceedings of LFCS2022},
       pages={pp.\ 18},
}

\bib{dagsamX}{article}{
      author={{Normann, Dag}},
      author={{Sanders, Sam}},
       title={On the uncountability of $\mathbb{R}$},
        date={2022},
     journal={To appear in the Journal of Symbolic Logic, arxiv:
  \url{https://arxiv.org/abs/2007.07560}},
       pages={pp.\ 40},
}

\bib{yamayamaharehare}{article}{
      author={{Sakamoto, Nobuyuki}},
      author={{Yamazaki, Takeshi}},
       title={Uniform versions of some axioms of second order arithmetic},
        date={2004},
     journal={MLQ Math. Log. Q.},
      volume={50},
      number={6},
       pages={587\ndash 593},
}

\bib{scheeffer}{article}{
      author={{Scheeffer, Ludwig}},
       title={Allgemeine untersuchungen \"{u}ber rectification der curven},
    language={German},
        date={1884},
     journal={Acta Math.},
      volume={5},
      number={1},
       pages={49\ndash 82},
}

\bib{simpson1}{book}{
      editor={{Simpson, Stephen G.}},
       title={Reverse mathematics 2001},
      series={Lecture Notes in Logic},
   publisher={ASL},
     address={La Jolla, CA},
        date={2005},
      volume={21},
}

\bib{simpson2}{book}{
      author={{Simpson, Stephen G.}},
       title={Subsystems of second order arithmetic},
     edition={2},
      series={Perspectives in Logic},
   publisher={Cambridge University Press},
        date={2009},
}

\bib{stillebron}{book}{
      author={{Stillwell, John}},
       title={Reverse mathematics, proofs from the inside out},
   publisher={Princeton Univ.\ Press},
        date={2018},
}

\bib{troeleke1}{book}{
      author={{Troelstra, Anne S.}},
      author={{van Dalen, Dirk}},
       title={Constructivism in mathematics i},
      series={Studies in Logic and the Foundations of Mathematics},
   publisher={North-Holland},
        date={1988},
      volume={121},
}

\bib{veldje2}{incollection}{
      author={{Veldman, Wim}},
       title={Understanding and using brouwer's continuity principle},
        date={2001},
       pages={285\ndash 302},
}

\bib{wica}{article}{
      author={{Wikipedia, The Free Encyclopedia}},
       title={Cantor's first set theory article},
        date={2022},
     journal={Website:
  \url{https://en.wikipedia.org/wiki/Cantor\%27s_first_set_theory_article}},
}

\end{biblist}
\end{bibdiv}
}

\end{document}